\documentclass{amsart}

\usepackage{amsmath}
\usepackage{amsthm}
\usepackage{amsfonts}
\usepackage{amssymb}
\usepackage{enumerate}

\newcommand{\indicator}[1]{\ensuremath{\mathbf{1}_{\{#1\}}}}
\newcommand{\oindicator}[1]{\ensuremath{\mathbf{1}_{#1}}}

\newcommand{\E}{\mathbb{E}}
\newcommand{\Prob}{\mathbb{P}}

\newcommand{\R}{\mathbb{R}}

\newcommand{\var}{\operatorname{Var}}

\newcommand{\tr}{\mathrm{tr}}

\newcommand{\eps}{\varepsilon}

\theoremstyle{plain}
  \newtheorem{theorem}{Theorem}

  \newtheorem{lemma}[theorem]{Lemma}

\theoremstyle{definition}
  \newtheorem{definition}[theorem]{Definition}
  
  \newtheorem{remark}[theorem]{Remark}

\begin{document}
\title[Partial Linear Eigenvalue Statistics]{Partial Linear Eigenvalue Statistics for Wigner and Sample Covariance Random Matrices}

\author[S. O'Rourke]{Sean O'Rourke}
\address{Department of Mathematics, Yale University, New Haven , CT 06520, USA  }
\email{sean.orourke@yale.edu}

\author[A. Soshnikov]{Alexander Soshnikov}
\address{Department of Mathematics, University of California, Davis, One Shields Avenue, Davis, CA 95616-8633  }
\email{soshniko@math.ucdavis.edu}

\begin{abstract}
Let $M_n$ be a $n \times n$ Wigner or sample covariance random matrix, and let $\mu_1(M_n), \mu_2(M_n), \ldots, \mu_n(M_n)$ denote the unordered eigenvalues of $M_n$.  We study the fluctuations of the partial linear eigenvalue statistics 
$$ \sum_{i=1}^{n-k} f(\mu_i(M_n)) $$
as $n \rightarrow \infty$ for sufficiently nice test functions $f$.  We consider both the case when $k$ is fixed and when $\min\{k,n-k\}$ tends to infinity with $n$.  
\end{abstract}

\maketitle

\section{Introduction}

We consider two classic random matrix ensembles with independent entries.  

\subsection{Wigner Random Matrices}

\begin{definition}[Wigner random matrix] \label{def:wigner}
We say $M_n = \frac{1}{\sqrt{n}} W_n = \frac{1}{\sqrt{n}} (w_{nij})_{1 \leq i,j \leq n}$ is a real symmetric (Hermitian) Wigner matrix of size $n$ if $M_n$ is a $n \times n$ real symmetric (Hermitian) matrix that satisfies the following.
\begin{enumerate}[(i)]
\item $\{w_{nij} : 1 \leq i \leq j \leq n \}$ is a collection of independent random variables,
\item for $1 \leq i < j \leq n$, $w_{nij}$ has zero mean and unit variance,
\item for $1 \leq i \leq n$, $w_{nii}$ has zero mean and variance $\sigma^2$.  
\end{enumerate}
\end{definition}

For a Wigner matrix $M_n$ of size $n$, we let $\lambda_1(M_n) \leq \lambda_2(M_n) \leq \cdots \leq \lambda_n(M_n)$ denote the \textit{ordered} eigenvalues of $M_n$ and let $\mu_1(M_n), \mu_2(M_n), \ldots, \mu_n(M_n)$ denote the \textit{unordered} eigenvalues of $M_n$.  That is, $\mu_i(M_n) = \lambda_{\pi(i)}(M_n)$ for $1 \leq i \leq n$, where $\pi$ is a random permutation on $\{1,2,\ldots,n\}$, chosen uniformly, independent of $M_n$.  

We will be interested in sequences of Wigner random matrices $\{M_n\}_{n \geq 1}$ that satisfy the following condition.  

\begin{definition}[Condition {\bf C0}]
For each $n \geq 1$, let $M_n$ be a real symmetric (Hermitian) Wigner matrix of size $n$.  We say the sequence $\{M_n\}_{n \geq 1}$ satisfies condition {\bf C0} with exponent $p \geq 0$ if there exists $\eps > 0$ such that
$$ \lim_{n \rightarrow \infty} n^{p/2} \left( \frac{n^{4 \eps}}{n^2} \sum_{1 \leq i < j \leq n} \E \left[ w_{nij}^4 \indicator{|w_{nij}| > n^{1/2 - \eps}} \right] + \frac{n^{2 \eps}}{n} \sum_{i=1}^n \E\left[ w_{nii}^2 \indicator{|w_{nii}| > n^{1/2 - \eps}}\right] \right) = 0, $$
where $\oindicator{E}$ denotes the indicator function of the event $E$.  
\end{definition}

\begin{remark}
For each $n \geq 1$, let $M_n = \frac{1}{\sqrt{n}} W_n = \frac{1}{\sqrt{n}} (w_{nij})_{1 \leq i, j \leq n}$ be a real symmetric (Hermitian) Wigner matrix of size $n$.  We note that if there exists $\delta > 0$ such that
$$ \sup_{n \geq 1, 1 \leq i < j \leq n} \E|w_{nij}|^{4+p+\delta} < \infty \quad \text{and} \quad \sup_{n \geq 1, 1 \leq i \leq n} \E|w_{nii}|^{2+p+\delta} < \infty, $$
then $\{M_n\}_{n \geq 1}$ satisfies condition {\bf C0} with exponent $p$.  We will mostly be interested in the cases when $p=0,1$.  
\end{remark}

\subsection{Sample Covariance Random Matrices}

\begin{definition}[Sample covariance matrix] \label{def:sc}
Let $A_n = \frac{1}{n} X_n^\ast X_n$ be an $n \times n$ matrix where $X_n = (x_{nij})_{1 \leq i,j \leq n}$.  We say that $A_n$ is a real (complex) sample covariance matrix of size $n$ if $\{x_{nij} : 1 \leq i,j \leq n\}$ is a collection of real (complex) independent random variables each with zero mean and unit variance.  In the complex case, we also require $\E (x_{nij}^2 )= 0$ for all $1 \leq i,j \leq n$.  
\end{definition}

For a sample covariance matrix $A_n$ of size $n$, we let $\lambda_1(A_n) \leq \lambda_2(A_n) \leq \cdots \leq \lambda_n(A_n)$ denote the \textit{ordered} eigenvalues of $A_n$ and let $\mu_1(A_n), \mu_2(A_n), \ldots, \mu_n(A_n)$ denote the \textit{unordered} eigenvalues of $A_n$.  

We will be interested in sequences of sample covariance matrices $\{A_n\}_{n \geq 1}$ that satisfy the following condition.  

\begin{definition}[Condition {\bf C1}]
For each $n \geq 1$, let $A_n = \frac{1}{n} X_n^\ast X_n$ be a real (complex) sample covariance matrix of size $n$ where $X_n = (x_{nij})_{1 \leq i,j \leq n}$.  We say the sequence $\{A_n\}_{n \geq 1}$ satisfies condition {\bf C1} if the random variables $x_{nij}, 1 \leq i,j \leq n$ have symmetric distribution for all $n \geq 1$ and there exists a constant $C_1$ such that
$$ \sup_{n \geq 1, 1 \leq i,j \leq n} \E|x_{nij}|^p \leq (C_1 \sqrt{p})^p \quad \text{for all} \quad p \geq 1. $$
\end{definition}

\subsection{Known Results}

For a Hermitian $n \times n$ matrix $B$, the empirical spectral distribution (ESD) $F^{B}(x)$ of $B$ is given by 
$$ F^{B}(x) := \frac{1}{n} \# \left\{ 1 \leq i \leq n : \lambda_i(B) \leq x \right\}, $$
where $\lambda_1(B), \lambda_2(B), \ldots, \lambda_n(B)$ denote the eigenvalues of $B$.  
Here $\#S$ denotes the cardinality of the set $S$.  

A fundamental problem in random matrix theory is to determine the limiting distribution of the ESD as the size of the matrix tends to infinity.  In the 1950s, Wigner studied the limiting ESD for a large class of random Hermitian matrices whose entries on or above the diagonal are independent \cite{W}.  Under certain conditions, Wigner showed that the ESD of such a matrix converges to the semicircle law $F$ with density given by 
\begin{equation} \label{eq:def:sc}
\rho(x):= \left\{
     \begin{array}{lr}
       \frac{1}{2\pi}\sqrt{4-x^2} , &-2 \leq x \leq 2 \\
       0 ,& \text{otherwise}.
     \end{array}
   \right. 
\end{equation}
   
The most general form of Wigner's semicircle law assumes only the first two moments of the entries \cite[Theorem 2.9]{BSbook}.

\begin{theorem}[Wigner's semicircle law] \label{thm:sclaw}
For each $n \geq 1$, let $M_n = \frac{1}{\sqrt{n}} W_n = \frac{1}{\sqrt{n}} (w_{nij})_{1 \leq i,j \leq n}$ be a real symmetric (Hermitian) Wigner matrix of size $n$.  Assume that for any $\varepsilon >0$, 
$$ \lim_{n \rightarrow \infty} \frac{1}{n^2} \sum_{i.j=1}^n \E|w_{nij}|^2 \indicator{|w_{nij}| \geq \varepsilon \sqrt{n}} = 0. $$
Then the ESD of $M_n$ converges to the semicircle law $F$ with density $\rho$ defined in \eqref{eq:def:sc}, almost surely as $n \rightarrow \infty$.  Equivalently, for any continuous, bounded function $f$,
$$ \frac{1}{n} \sum_{i=1}^n f(\lambda_i(M_n)) \longrightarrow \int_{-\infty}^{\infty} f(x) \rho(x) dx $$
almost surely as $n \rightarrow \infty$.  
\end{theorem}

The sample covariance case was studied by Marchenko and Pastur \cite{MP}.  In particular, they showed that, under certain conditions, the ESD of a sample covariance random matrix converges to $F_{\mathrm{MP}}$ with density given by
\begin{equation} \label{eq:def:mp}
	\rho_{\mathrm{MP}}(x) := \left\{
     		\begin{array}{lr}
       		\frac{1}{2\pi} \sqrt{ \frac{4-x}{x}}, & 0 < x < 4 \\
       		0, &  \text{otherwise}.
     		\end{array}
   \right. 
\end{equation}

The Marchenko-Pastur law is the limiting ESD for a large class of sample covariance random matrices \cite[Theorem 3.10]{BSbook}.

\begin{theorem}[Marchenko-Pastur law] \label{thm:mplaw}
For each $n \geq 1$, let $A_n = \frac{1}{n} X_n^\ast X_n$ be a real (complex) sample covariance of size $n$, where $X_n = (x_{nij})_{1 \leq i,j \leq n}$.  Assume that for any $\varepsilon > 0$, 
$$ \lim_{n \rightarrow \infty} \frac{1}{n^2} \sum_{i,j=1}^n \E|x_{nij}|^2 \indicator{|x_{nij}| \geq \varepsilon \sqrt{n}} = 0. $$
Then the ESD of $A_n$ converges to the Marchenko-Pastur law $F_{\mathrm{MP}}$ almost surely as $n \rightarrow \infty$.  Equivalently, for any continuous, bounded function $f$,
$$ \frac{1}{n} \sum_{i=1}^n f(\lambda_i(A_n)) \longrightarrow \int_{-\infty}^\infty f(x) \rho_{\mathrm{MP}}(x) dx $$
almost surely as $n \rightarrow \infty$.  
\end{theorem}

Theorems \ref{thm:sclaw} and \ref{thm:mplaw} can be viewed as random matrix theory analogues of the Law of Large Numbers from classical probability theory. Thus a Central Limit Theorem for fluctuations of linear eigenvalue statistics is a natural next step.  

In \cite{S}, Shcherbina studies the fluctuations of linear eigenvalue statistics for both Wigner and sample covariance random matrices.  In particular, she considers test functions $f$ from the space $\mathcal{H}_s$ with the norm
$$ \|f\|^2_s := \int (1 + 2|l|)^{2s}|\hat{f}(l)|^2 dl $$
for $s > 3/2$, where $\hat{f}$ is the Fourier transform of $f$ defined by
$$ \hat{f}(l) := \frac{1}{\sqrt{2 \pi}} \int e^{ilx} f(x) dx. $$
We note that if $f$ is a real-valued function with $f \in \mathcal{H}_s$ for $s > 3/2$, then both $f$ and $f'$ are continuous and bounded almost everywhere \cite{HN}.  In particular, this implies that $f$ is Lipschitz.  

Shcherbina obtains the following two results \cite[Theorems 1 and 2]{S}.  

\begin{theorem}[Linear eigenvalue statistics for Wigner matrices; \cite{S}] \label{thm:wigner:linear}
For each $n \geq 1$, let $M_n = \frac{1}{\sqrt{n}} W_n = \frac{1}{\sqrt{n}} (w_{nij})_{1 \leq i,j \leq n}$ be a real symmetric Wigner matrix of size $n$.  Suppose $\E[w_{nij}^4] = m_4$ for all $1 \leq i < j \leq n$ and all $n \geq 1$.  Assume for any $\varepsilon > 0$,
$$ \lim_{n \rightarrow \infty} \left( \frac{1}{n} \sum_{i = 1}^n \E|w_{nii}|^2 \indicator{|w_{nii}| \geq \varepsilon \sqrt{n}} + \frac{1}{n^2} \sum_{1 \leq i < j \leq n} \E |w_{nij}|^4 \indicator{|w_{nij}| \geq \varepsilon \sqrt{n}} \right) = 0. $$
Let $f$ be a real-valued function with $\|f\|_s < \infty$ for some $s > 3/2$.  Then
$$ \sum_{i=1}^n f(\lambda_i(M_n)) - \E \sum_{i=1}^n f(\lambda_i(M_n)) \longrightarrow N(0,v^2[f]) $$
in distribution as $n \rightarrow \infty$, where
\begin{align} \label{eq:def:v2f}
	v^2[f] &:= \frac{1}{2\pi^2} \int_{-2}^2 \int_{-2}^2 \left( \frac{f(x) - f(y)}{x-y} \right)^2 \frac{ 4-xy}{\sqrt{4-x^2}\sqrt{4-y^2} } dx dy \\
	&\qquad + \frac{m_4 - 3}{2 \pi^2} \left( \int_{-2}^2 f(x) \frac{2 - x^2}{\sqrt{4-x^2}} dx \right)^2 + \frac{\sigma^2 - 2}{4 \pi^2} \left( \int_{-2}^2 \frac{f(x) x}{\sqrt{4 - x^2}} dx  \right)^2. \nonumber
\end{align}
\end{theorem}

\begin{theorem}[Linear eigenvalue statistics for sample covariance matrices; \cite{S}] \label{thm:sc:linear}
For each $n \geq 1$, let $A_n = \frac{1}{n} X_n^\ast X_n$ be a real sample covariance matrix of size $n$, where $X_n = (x_{nij})_{1 \leq i,j \leq n}$.  Suppose $\E[x_{nij}^4] = m_4$ for all $1 \leq i, j \leq n$ and all $n \geq 1$.  Assume there exists $\varepsilon > 0$ such that
$$ \sup_{n \geq 1} \sup_{1 \leq i,j \leq n} \E|x_{nij}|^{4 + \eps} < \infty. $$
Let $f$ be a real-valued function with $\|f\|_s < \infty$ for some $s > 3/2$.  Then
$$ \sum_{i=1}^n f(\lambda_i(A_n)) - \E \sum_{i=1}^n f(\lambda_i(A_n)) \longrightarrow N(0,v_{\mathrm{SC}}^2[f]) $$
in distribution as $n \rightarrow \infty$, where
\begin{align} \label{eq:def:vsc2f}
	v_{\mathrm{SC}}^2[f] &:= \frac{1}{2 \pi^2} \int_{0}^4 \int_0^4 \left( \frac{ f(x) - f(y)}{x-y} \right)^2 \frac{ \left( 4 - (x-2)(y-2) \right) }{ \sqrt{ 4 - (x-2)^2} \sqrt{4 - (y-2)^2}} dx dy \\
	& \qquad + \frac{m_4 - 3}{4 \pi^2} \left( \int_0^4 \frac{x - 2}{\sqrt{4 - (x-2)^2}} dx \right)^2. \nonumber
\end{align}

\end{theorem}

Analogous results for other random matrix ensembles (and other classes of test functions $f$) have also been obtained; see for example 
\cite{AZ,BS,DS,DE,J,LP,S,SS,Sclt,SoWa} and references therein.

\subsection{Main Results}

One can observe from Theorems \ref{thm:wigner:linear} and \ref{thm:sc:linear} that the variance of the linear eigenvalue statistics does not grow to infinity in the limit $n \to \infty$ for sufficiently smooth test functions.  This points to very effective cancellations between different terms of the sum and a rigidity property for the distribution of the eigenvalues.  

In \cite{J88}, K. Johansson 
considered, among other things, linear statistics 
\begin{equation}
\label{lin}
S_n=\sum_{j=1}^n g(\theta_j)
\end{equation}
of the eigenvalues $e^{i\*\theta_1}, \ldots, e^{i\*\theta_n}$ of a 
random $n\times n$ unitary matrix distributed according to Haar measure on $U(n).$  He proved the 
CLT for such linear statistics under the optimal 
condition $$\sum_{-\infty}^{\infty} |k|\*|\hat{g}_k|^2 <\infty, $$ 
where $\hat{g}_k$ are the Fourier coefficients of $g,$
and also connected the CLT result to the Szeg\"o asymptotic formula for Toeplitz determinants (see e.g. \cite{GI, Ka, Sz}).  
Johansson's proof relies on elaborate cancellations.  In particular, it works with minor modifications 
for more general $\beta$ ensembles, $\beta>0$
(the Haar measure case corresponding to $\beta=2$).

In remark 2.1 of \cite{J88}, Johansson noted that for $ S_{n,1}= \sin \theta_1+ \ldots \sin \theta_{n-1}, $ the sum of the first $n-1$ terms in 
(\ref{lin}) with $g(\theta)=\sin \theta,$ the distribution of the normalized statistic
$$ \frac{S_{n,1}-\E S_{n,1}}{\sqrt{\var S_{n,1}}}$$
does not converge to the standard normal distribution as $n\to \infty.$  The argument relies on the CLT 
for $S_n$ and the following facts: $\var S_n= 1/2, \ 
\var S_{n,1}\to 1, $ as $n\to \infty,$ and $|S_n-S_{n,1}|\leq 1.$

In this paper, we study the fluctuations of the partial linear eigenvalue statistics 
\begin{equation} \label{eq:def:snk}
	S_{n,k}[f] := \sum_{i=1}^{n-k} f(\mu_i(M_n)) 
\end{equation}
where $k = k(n)$ is a positive integer sequence, $f$ is a sufficiently nice test function from $\mathcal{H}_s, s > 3/2$, and  
$\{M_n\}_{n \geq 1}$ is a sequence of Wigner matrices that satisfy condition {\bf C0}.

\begin{theorem} \label{thm:wigner:finite}
For each $n \geq 1$, let $M_n = \frac{1}{\sqrt{n}} W_n = \frac{1}{\sqrt{n}} (w_{nij})_{1 \leq i,j \leq n}$ be a real symmetric Wigner matrix of size $n$.  Assume the sequence $\{M_n\}_{n \geq 1}$ satisfies condition {\bf C0} with exponent $0$ and suppose $\E[w_{nij}^4] = m_4$ for all $1 \leq i < j \leq n$ and all $n \geq 1$.  Let $f$ be a real-valued, bounded Lipschitz function with $\|f\|_s < \infty$ for some $s > 3/2$.  Let $k$ be a fixed positive integer and let $S_{n,k}[f]$ be defined by \eqref{eq:def:snk}.  Then
$$ S_{n,k}[f] - \E S_{n,k}[f]  \longrightarrow N(0,v^2[f]) * \left[ - \sum_{i=1}^k \left[ f(\psi_i) - \E f(\psi_i) \right] \right], $$
in distribution as $n \rightarrow \infty$, where $\psi_1, \ldots, \psi_k$ are i.i.d. semicircle-distributed random variables and $v^2[f]$ is given in \eqref{eq:def:v2f}.  
\end{theorem}

\begin{theorem} \label{thm:wigner:infty}
For each $n \geq 1$, let $M_n = \frac{1}{\sqrt{n}} W_n = \frac{1}{\sqrt{n}} (w_{nij})_{1 \leq i,j \leq n}$ be a real symmetric Wigner matrix of size $n$.  Assume the sequence $\{M_n\}_{n \geq 1}$ satisfies condition {\bf C0} with exponent $1$ and suppose $\E[w_{nij}^4] = m_4$ for all $1 \leq i < j \leq n$ and all $n \geq 1$.  Let $f$ be a real-valued, bounded Lipschitz function with $\|f\|_s < \infty$ for some $s > 3/2$.  Let $k = k(n)$ be a positive integer sequence such that $\min \{k,n-k \} \rightarrow \infty$ as $n \rightarrow \infty$.  Let $S_{n,k}[f]$ be defined by \eqref{eq:def:snk}.  Then 
$$ \alpha_{n,k} \left( S_{n,k}[f] - \E S_{n,k}[f] \right) \longrightarrow N(0,d^2[f]) $$
in distribution as $n \rightarrow \infty$, where 
\begin{equation} \label{def:alpha}
	\alpha_{n,k} := \sqrt{\frac{{n}}{{k(n-k)}} }
\end{equation}
and $d^2[f] := \var[f(\psi)]$ for a semicircle-distributed random variable $\psi$.  
\end{theorem}

\begin{remark}
One can also study the case when $n-k$ is a fixed positive integer.  In the proof of Theorem \ref{thm:wigner:finite} below, we show that if $l$ is a fixed positive integer
$$ \sum_{i=1}^l f(\mu_i(M_n)) \longrightarrow \sum_{i=1}^l f(\psi_i) $$
in distribution as $n \rightarrow \infty$, where $\psi_1, \ldots, \psi_l$ are i.i.d. semicircle-distributed random variables.  
\end{remark}

It should be mentioned that a different type of partial linear eigenvalue statistic for Wigner matrices has been recently 
studied by Bao, Pan, and Zhou in \cite{BPZ}.   In particular, they consider $\sum_{i=1}^{k} f(\lambda_i(M_n)),$ 
where $\lambda_1(M_n)\leq \lambda_2(M_n)\ldots \leq \lambda_n(M_n)$ are the ordered eigenvalues of $M_n$ and $k$ is proportional to $n$.

Now we turn our attention to sample covariance random matrices.
Let $\{A_n\}_{n \geq 1}$ be a sequence of sample covariance matrices that satisfies condition {\bf C1}.  In this case, we consider the partial linear eigenvalue statistics
\begin{equation} \label{eq:def:tnk}
	T_{n,k}[f] := \sum_{i=1}^{n-k} f(\mu_i(A_n)).
\end{equation}

\begin{theorem} \label{thm:sc:finite}
For each $n \geq 1$, let $A_n = \frac{1}{n} X_n^\ast X_n$ be a real sample covariance matrix of size $n$ where $X_n = (x_{nij})_{1 \leq i,j \leq n}$.  Assume the sequence $\{A_n\}_{n \geq 1}$ satisfies condition {\bf C1} and suppose $\E[x_{nij}^4] = m_4$ for all $1 \leq i,j \leq n$ and all $n \geq 1$.  Let $f$ be a real-valued, bounded Lipschitz function with $\|f\|_s < \infty$ for some $s > 3/2$.  Let $k$ be a fixed positive integer and let $T_{n,k}[f]$ be given by \eqref{eq:def:tnk}.  Then
$$ T_{n,k}[f] - \E T_{n,k}[f]  \longrightarrow N(0,v_{\mathrm{SC}}^2[f]) * \left[ - \sum_{i=1}^k \left[ f(\psi_i) - \E f(\psi_i) \right] \right], $$
in distribution as $n \rightarrow \infty$, where $\psi_1, \ldots, \psi_k$ are i.i.d. Marchenko-Pastur distributed random variables and $v_{\mathrm{SC}}^2[f]$ is given by \eqref{eq:def:vsc2f}.  
\end{theorem}

\begin{theorem} \label{thm:sc:infty}
For each $n \geq 1$, let $A_n = \frac{1}{n} X_n^\ast X_n$ be a real sample covariance matrix of size $n$ where $X_n = (x_{nij})_{1 \leq i,j \leq n}$.  
Assume the sequence $\{A_n\}_{n \geq 1}$ satisfies condition {\bf C1} and suppose $\E[x_{nij}^4] = m_4$ for all $1 \leq i,j \leq n$ and all $n \geq 1$.  
Let $f$ be a real-valued, Lipschitz function with $\|f\|_s < \infty$ for some $s > 3/2$.  Let $k = k(n)$ be a positive integer sequence such that 
$\min \{k,n-k \} \rightarrow \infty$ as $n \rightarrow \infty$.  Let $T_{n,k}[f]$ be given by \eqref{eq:def:tnk}.  Then 
$$ \alpha_{n,k} \left( T_{n,k}[f] - \E T_{n,k}[f] \right) \longrightarrow N(0,d_{\mathrm{SC}}^2[f]) $$
in distribution as $n \rightarrow \infty$, where $\alpha_{n,k}$ is defined in \eqref{def:alpha} and $d_{\mathrm{SC}}^2[f] := \var[f(\psi)]$ for a 
Marchenko-Pastur distributed random variable $\psi$.  
\end{theorem}

The last two theorems are valid under assumptions weaker than condition {\bf C1}, since one can derive the local Marchenko-Pastur law\footnote{The conclusion of the Marchenko-Pastur law (Theorem \ref{thm:mplaw}) can be equivalently stated as $$ \frac{\#\{1 \leq i \leq n : \lambda(A_n) \in I\}}{n} \longrightarrow \int_I \rho_{\mathrm{MP}} (x) dx $$ almost surely as $n \rightarrow \infty$, for any fixed interval $I$.  The local Marchenko-Pastur law refers to a similar conclusion holding when the interval $I$ is allowed to change with $n$.  Of particular interest is the case when the length of the interval decreases as $n$ tends to infinity; see for instance \cite{CMS} and references therein. } at the optimal scale under 
assumptions analogous to those in condition {\bf C0} (\cite{E}).

\subsection{Notation and Overview}
Asymptotic notations such as $O, o, \Omega$, and so forth, are used under the assumption that $n \rightarrow \infty$.  The notation $O_{C}(\cdot)$ 
emphasizes that the hidden constant depends on $C$.  

An event $E$, which depends on $n$, is said to hold with overwhelming probability if $\Prob(E) \geq 1 - O_C(n^{-C})$ for every constant $C > 0$.  
We let $E^C$ denote the complement of the event $E$.  

The paper is organized as follows.  In Section \ref{sec:wigner:proof}, we prove Theorems \ref{thm:wigner:finite} and \ref{thm:wigner:infty}.  
Section \ref{sec:sc:proof} is devoted to Theorems \ref{thm:sc:finite} and \ref{thm:sc:infty}.  

\subsection{Acknowledgements}
The authors would like to thank Persi Diaconis and Laszlo Erd\"os for useful comments.  We also thank the anonymous referee for many helpful comments, corrections, and references.  A.S. has been supported in part by the NSF grant DMS-1007558.  S.O. has been supported by grant AFOSAR-FA-9550-12-1-0083.  

\section{Proof of Theorems \ref{thm:wigner:finite} and \ref{thm:wigner:infty}} \label{sec:wigner:proof}

In order to study the limiting distribution of $S_{n,k}[f] - \E S_{n,k}[f]$, we let $(\xi_1, \ldots, \xi_k)$ be a random sample without replacement 
from $\{1,2,\ldots, n\}$ independent of $M_n$.  Then 
$$ S_{n,k}[f] \stackrel{\text{d}}{=} \sum_{i=1}^{n} f(\lambda_i(M_n)) - \sum_{j=1}^k f(\lambda_{\xi_j}(M_n)) = L_n[f] - \sum_{j=1}^k f(\lambda_{\xi_j}(M_n)), $$
where
$$ L_n[f] := \tr \left[ f(M_n) \right]. $$

We now take advantage of the following rigidity result based on \cite[Theorem 2.2]{EYY} and \cite[Theorem 3.6]{LY}.  Let $\eta_j = \eta^{(n)}_j$ be the classical location of the $j$th eigenvalue.  That is,
$$ \int_{-\infty}^{\eta_j}\rho(x) dx = \frac{j}{n} $$
where $\rho$ is the density of the semicircle distribution given in \eqref{eq:def:sc}.  

\begin{theorem}[Rigidity of eigenvalues] \label{thm:wigner:rigidity}
Let $M_n= \frac{1}{\sqrt{n}} W_n = \frac{1}{\sqrt{n}} (w_{ij})_{1 \leq i, j \leq n}$ be a real symmetric Wigner matrix of size $n$.  Assume there exists a constant $C_1$ such that
\begin{equation} \label{eq:4mom}
	\sup_{1 \leq i < j \leq n} \E[w_{ij}^4] \leq C_1. 
\end{equation}
Then for any $0 < \eps < 1/2$, there exists constants $C, c > 0$ and $n_0$ (depending only on $C_1, \eps$, and $\sigma$ from Definition \ref{def:wigner}) such that the event
\begin{equation} \label{eq:wigner:event}
	\left\{  \exists j : |\lambda_j(M_n) - \eta_j| \geq (\log n)^{c \log \log n} n^{-2/3} [\min\{j,n-j+1\}]^{-1/3} \right\}
\end{equation}
holds with probability at most
$$ \Prob(\Omega_n^C) + C\left( \frac{n^{4 \eps}}{n^2} \sum_{1 \leq i < j \leq n} \E\left[ w_{ij}^4 \indicator{|w_{ij}| > n^{1/2 - \eps}}\right] + \frac{n^{2 \eps}}{n} \sum_{i=1}^n \E \left[ w_{ii}^2 \indicator{|w_{ii}| > n^{1/2 - \eps}} \right] \right) $$
for any $n>n_0$, where the event $\Omega_n$ holds with overwhelming probability.  
\end{theorem}

The proof of Theorem \ref{thm:wigner:rigidity} is based on the machinery developed in \cite{EYY,LY}; we present the proof in Appendix \ref{sec:wigner:rigidity:proof}.  For the moment, we assume Theorem \ref{thm:wigner:rigidity} and complete the proof of Theorems \ref{thm:wigner:finite} and \ref{thm:wigner:infty}.

It follows from Theorem \ref{thm:wigner:rigidity} that
\begin{equation} \label{eq:fxi}
	\sum_{j=1}^k f(\lambda_{\xi_j}(M_n)) = \sum_{j=1}^k f(\eta_{\xi_j}) + O_f \left(\frac{k (\log n)^{c \log \log n}}{n^{2/3}} \right) 
\end{equation}
with probability $1-o(1)$.  Here we have used the fact that $f$ is Lipschitz.  

\begin{proof}[Proof of Theorem \ref{thm:wigner:finite}]
From \eqref{eq:fxi}, we have that 
$$ L_n[f] -  \sum_{j=1}^k f(\lambda_{\xi_j}(M_n)) = L_n[f] - \sum_{j=1}^k f(\eta_{\xi_j}) + o_f(1) $$
with probability $1-o(1)$.  We note that $L_n[f]$ and $\sum_{j=1}^k f(\eta_{\xi_j})$ are independent.  

It also follows from Theorem \ref{thm:wigner:rigidity} that
$$ \E S_{n,k}[f] = \E L_n[f] - \E \sum_{j=1}^k f(\eta_{\xi_j}) + o_f(1)$$
since $f$ is bounded.

By Theorem \ref{thm:wigner:linear}, it follows that $L_n[f] - \E L_n[f]$ converges to a normal distribution with mean zero and variance $v^2[f]$.  It remains to compute the limiting distribution of 
$$ \sum_{j=1}^k f(\eta_{\xi_j}) - \E \sum_{j=1}^k f(\eta_{\xi_j}). $$

Let $\tilde{\xi}_1, \ldots, \tilde{\xi}_k$ be i.i.d. uniform random variables on $\{1,2,\ldots, n\}$ independent of $M_n$.  We begin by noting that
$$ \E \sum_{j=1}^k f(\eta_{\xi_j}) = \E \sum_{j=1}^k f(\eta_{\tilde{\xi_j}}). $$

Let $g$ be an arbitrary bounded, continuous function.  Then
\begin{align*}
	\E g \left( \sum_{j=1}^k f(\eta_{\xi_j}) \right) &= \sum_{i_1, \ldots i_k \text{ distinct}} g \left( \sum_{j=1}^k f(\eta_{i_j}) \right) \frac{1}{n(n-1) \cdots (n-k+1)} \\
	&= \sum_{i_1, \ldots i_k \text{ distinct}} g \left( \sum_{j=1}^k f(\eta_{i_j}) \right) \frac{1}{n^k} + O_g \left( \frac{k^2}{n} \right) \\
	&= \sum_{i_1, \ldots, i_k =1}^n g \left( \sum_{j=1}^k f(\eta_{i_j}) \right) \frac{1}{n^k} + O_g \left( \frac{k^2}{n} \right) \\
	&= \E g \left( \sum_{j=1}^k f(\eta_{\tilde{\xi}_j}) \right) + O_g \left( \frac{k^2}{n} \right).
\end{align*}

Therefore the limiting distribution of $\sum_{j=1}^k f(\eta_{\xi_j})$ is the same as the limiting distribution of $\sum_{j=1}^k f(\eta_{\tilde{\xi}_j})$.  A simple computation reveals that $\eta_{\tilde{\xi}_j}$ converges to the semicircle distribution as $n \rightarrow \infty$.  Since $f$ is continuous and bounded on $[-2,2]$, and $k$ is fixed, the proof of Theorem \ref{thm:wigner:finite} is complete.  
\end{proof}  

\begin{proof}[Proof of Theorem \ref{thm:wigner:infty}]
Suppose $\min\{k,n-k\} \rightarrow \infty$ as $n \rightarrow \infty$.  We begin by noting that
$$ S_{n,k}[f] \stackrel{\text{d}}{=} \sum_{i=1}^{n-k} f(\lambda_{\xi_i}(M_n)) \stackrel{\text{d}}{=} L_n[f] - \sum_{i=1}^k f(\lambda_{\xi_i}(M_n)), $$
where $(\xi_1,\xi_2,\ldots,\xi_n)$ is a random sample without replacement from $\{1,2,\ldots,n\}$ independent of $M_n$.  

Since $\alpha_{n,k} = o(1)$, it follows from Theorem \ref{thm:sc:linear} that 
$$ \alpha_{n,k} \left( L_n[f] - \E L_n[f] \right) \longrightarrow 0 $$
in probability as $n \rightarrow \infty$.  Therefore, it suffices to show that
\begin{equation} \label{eq:aflshow}
	\alpha_{n,k} \sum_{i=1}^k [f(\lambda_{\xi_i}(M_n)) - \E f(\lambda_{\xi_i}(M_n))] \longrightarrow N(0,d^2[f]) 
\end{equation}
in distribution as $n \rightarrow \infty$ or
\begin{equation} \label{eq:aflshow2}
	\alpha_{n,k} \sum_{i=1}^{n-k} [f(\lambda_{\xi_i}(M_n)) - \E f(\lambda_{\xi_i}(M_n))] \longrightarrow N(0,d^2[f])
\end{equation}
in distribution as $n \rightarrow \infty$.  

We will verify \eqref{eq:aflshow} when $k \leq n-k$ and verify \eqref{eq:aflshow2} in the case when $k > n-k$.  In the setting where the sequence $\{k(n)\}_{n \geq 1}$ alternates between the two cases, we use a sub-sequence argument since the limit in each case will be the same.  

Since the argument is the same in each case, we assume $k \leq n-k$ and verify \eqref{eq:aflshow}.  In this case, $\alpha_{n,k} = O(k^{-1/2})$.  From condition {\bf C0}, we find that the event \eqref{eq:wigner:event} from Theorem \ref{thm:wigner:rigidity} holds with probability $o(n^{-1/2})$.  Since $f$ is bounded, it follows that
\begin{equation} \label{eq:expank}
	\alpha_{n,k} \sum_{i=1}^k  \E f(\lambda_{\xi_i}(M_n)) = \alpha_{n,k} \sum_{i=1}^k  \E f(\eta_{\xi_i}) + o_f(1) 
\end{equation}

Therefore, by \eqref{eq:fxi} and \eqref{eq:expank}, it suffices to show that
\begin{equation} \label{eq:afxishow}
	\alpha_{n,k} \sum_{i=1}^k [f(\eta_{\xi_i}) - \E f(\eta_{\xi_i})] \longrightarrow N(0,d^2[f]) 
\end{equation}
in distribution as $n \rightarrow \infty$.  \eqref{eq:afxishow} will follow from Lemma \ref{lemma:depclt} below.  Indeed, since $f$ is bounded, a simple computation reveals that
$$ \var[f(\eta_{\xi_1})] \longrightarrow \var[f(\psi)] $$
as $n \rightarrow \infty$, where $\psi$ is a semicircle-distributed random variable.  
\end{proof}

\begin{lemma} \label{lemma:depclt}
For each $n \geq 1$, let $(\xi^{(n)}_1, \ldots, \xi^{(n)}_k)$ be a discrete random sample on $[n]:=\{1,2,\ldots, n\}$, where $k = k(n)$ is a positive integer sequence such that $\min\{k,n-k\} \rightarrow \infty$ as $n \rightarrow \infty$. Define $\zeta_i^{(n)} := \frac{\xi_i^{(n)}}{n}$ for each $n \geq 1$ and $i=1,2,\ldots,k$.  Let $f: [0,1] \rightarrow \R$ be a bounded function.  Then 
$$ \alpha_{n,k} \sum_{i=1}^k [f({\zeta^{(n)}_i}) - \E f({\zeta^{(n)}_i})] \longrightarrow N(0,\beta^2) $$
in distribution as $n \rightarrow \infty$, where 
$$ \beta^2 := \lim_{n \rightarrow \infty} \var[f(\zeta^{(n)}_1)] $$
and $\alpha_{n,k}$ is defined in \eqref{def:alpha}.  
\end{lemma}

Lemma \ref{lemma:depclt} is a direct consequence of \cite[Theorem 1]{ER} (see also \cite{H} and \cite[Section 3]{Hbook}).  For completeness we give a proof of Lemma \ref{lemma:depclt} in Appendix \ref{sec:depclt}.

\section{Proof of Theorems \ref{thm:sc:finite} and \ref{thm:sc:infty}} \label{sec:sc:proof}

In order to prove Theorems \ref{thm:sc:finite} and \ref{thm:sc:infty} we require a rigidity estimate for the eigenvalues of sample covariance random matrices.  Theorem \ref{thm:sc:rigidity} below provides such an estimate and is similar to Theorem \ref{thm:wigner:rigidity}.  

Let $\gamma_j = \gamma^{(n)}_j$ be the classical location of the $j$th eigenvalue.  That is,
$$ \int_{0}^{\gamma_j}\rho_{\mathrm{MP}}(x) dx = \frac{j}{n} $$
where $\rho_{\mathrm{MP}}$ is the density of the Marchenko-Pastur law given in \eqref{eq:def:mp}.  

\begin{theorem} \label{thm:sc:rigidity}
Let $A_n = \frac{1}{n} X_n^\ast X_n$ be a real (complex) sample covariance matrix of size $n$ where $X_n = (x_{ij})_{1 \leq i,j \leq n}$.  Assume there exists a constant $C_1$ such that
\begin{equation} \label{eq:subgauss}
	\sup_{1 \leq i,j \leq n} \E|x_{ij}|^p \leq (C_1 \sqrt{p})^p \quad \text{for all} \quad p \geq 1
\end{equation}
and suppose $x_{ij}, 1 \leq i,j \leq n$ have symmetric distribution.  Then there exists constants $C,c,c_0,c_1>0$ (depending only on $C_1$) such that
$$ \Prob \left( \exists j : |\lambda_j(A_n) - \gamma_j| \geq C(\log n)^{c \log \log n} n^{-2/3} \right)  \leq C \exp \left(- c_o (\log n)^{c_1 \log \log n} \right) $$
for $n$ sufficiently large.
\end{theorem}

With this rigidity estimate in hand, the proof of Theorems \ref{thm:sc:finite} and \ref{thm:sc:infty} is nearly identical to the proof of Theorems \ref{thm:wigner:finite} and \ref{thm:wigner:infty}; we leave the details to the reader.  It remains to prove Theorem \ref{thm:sc:rigidity}.  

We will need the following version of \cite[Lemma 5.1]{BYY}.  It should be noted that \cite[Lemma 5.1]{BYY} is much more general than the version stated here.  For convenience, we define
$$ \varphi_n := (\log n)^{\log \log n}. $$

\begin{lemma}[\cite{BYY}] \label{lemma:byy:rigidity}
Let $A_n = \frac{1}{n} X_n^\ast X_n$ be a real (complex) sample covariance matrix of size $n$ where $X_n = (x_{ij})_{1 \leq i,j \leq n}$.  Assume there exists a constant $C_1$ such that \eqref{eq:subgauss} holds.  Then there exists constants $C,c,c_0>0$ (depending only on $C_1$) such that for any $\varphi_n^c < j < n - \varphi_n^c$, 
$$ \Prob \left( \gamma_{j - \varphi_n^c} \leq \lambda_j(A_n) \leq \gamma_{j + \varphi_n^c} \right) \geq 1 - \exp(- c_0 \varphi_n) $$
and
$$ \Prob \left( \frac{|\lambda_j(A_n) - \gamma_j|}{\gamma_j} > \frac{ C \varphi_n^c}{j \left( 1 - \frac{j}{n} \right)^{1/3}} \right) \leq \exp( -c_0 \varphi_n ). $$
\end{lemma}

\begin{proof}[Proof of Theorem \ref{thm:sc:rigidity}]
By the union bound, it suffices to show that 
$$ \Prob \left( | \lambda_j(A_n) - \gamma_j| \geq C \varphi_n^c n^{-2/3} \right) \leq C \exp (-c_0 \varphi_n^{c_1} ) $$
for each $1 \leq j \leq n$.  Let $0 < \varepsilon < 1/100$.  We consider several cases.  
\begin{enumerate}[(i)]
\item If $\varepsilon n \leq j < n - \varphi_n^c$, then
$$ |\lambda_j(A_n) - \gamma_j| \leq \frac{C \varphi_n^c \gamma_j}{\varepsilon n (\varphi_n^c/n)^{1/3} }\leq \frac{4C \varphi_n^c}{\varepsilon n^{2/3}}  $$
with probability at least $1 - \exp(-c_0 \varphi_n)$ by Lemma \ref{lemma:byy:rigidity}.  

\item Consider the case when $\varphi_n^c < j \leq \varepsilon n$.  Using that $j \leq \varepsilon n$, we have 
$$ \frac{j}{n} = \int_0^{\gamma_j} \rho_{\mathrm{MP}}(x) dx \geq \frac{ \sqrt{2}}{2 \pi} \int_0^{\gamma_j} x^{-1/2} dx = \frac{ \sqrt{2} }{ \pi} \sqrt{\gamma_j} $$
since $\gamma_j \leq 2$.  Thus, we obtain the bound
\begin{equation} \label{eq:gjpi}
	\gamma_j \leq \frac{\pi^2 j^2}{2n^2} \quad \text{ for all } 1 \leq j \leq \varepsilon n
\end{equation}
Since $\varphi_n^c < j \leq \varepsilon n$, we combine \eqref{eq:gjpi} with Lemma \ref{lemma:byy:rigidity} to obtain
$$ |\lambda_j(A_n) - \gamma_j| \leq \frac{ C \pi^2 \varphi_n^c j}{2 (1-\varepsilon)^{1/3} n^2} \leq \frac{ C \pi^2 \varepsilon}{2(1-\varepsilon)^{1/3}} \frac{\varphi_n^c}{n} $$
with probability at least  $1 - \exp(-c_0 \varphi_n)$.  

\item If $1 \leq j \leq \varphi_n^c$, then $\lambda_j \leq \gamma_{3 \varphi_n^c}$ with probability at least  $1 - \exp(-c_0 \varphi_n)$ by Lemma \ref{lemma:byy:rigidity}.  Using the bound \eqref{eq:gjpi} for $\gamma_{3 \varphi_n^c}$, we obtain 
$$ |\lambda_j(A_n) - \gamma_j| \leq 2 \gamma_{3 \varphi_n^c} \leq 9 \pi^2 \frac{ \varphi_n^{2c} }{ n^2} $$
with probability at least  $1 - \exp(-c_0 \varphi_n)$.  

\item Consider the final case when $n - \varphi_n^c \leq j \leq n$.  First we note that for any $k \leq 3 \varphi_n^c$, 
$$ \frac{3 \varphi_n^c}{n} \geq \frac{k}{n} = \int_{\gamma_{n-k}}^4 \rho_{\mathrm{MP}}(x) dx \geq \frac{1}{2 \sqrt{2} \pi} \int_{\gamma_{n-k}}^4 \sqrt{4-x} dx $$
and hence 
\begin{equation} \label{eq:4gnk}
	|4 - \gamma_{n-k}| \leq \left[ \frac{ 9 \sqrt{2} \pi \varphi_n^c}{n} \right]^{2/3} \quad \text{for all} \quad k \leq 3\varphi_n^c.  
\end{equation}
By Lemma \ref{lemma:byy:rigidity} and the estimate above, it follows that
\begin{equation} \label{eq:ljann}
	\lambda_j(A_n) \geq \gamma_{n - 3 \varphi_n^c} \geq 4 - \left[ \frac{ 9 \sqrt{2} \pi \varphi_n^c}{n} \right]^{2/3} 
\end{equation}
with probability at least  $1 - \exp(-c_0 \varphi_n)$.   By \cite[Lemma 3]{SOuniv} and Markov's inequality, there exists constant $C',c'>0$ such that
\begin{align*}
	\Prob \left( \lambda_n(A_n) \geq 4 + {\varphi_n^c}n^{-2/3} \right) &\leq \frac{ \E\left[ \tr (A_n)^{\lfloor n^{2/3} \rfloor} \right] }{ (4 + \varphi_n^c n^{-2/3})^{\lfloor n^{2/3} \rfloor}} \\
	&\leq C' \left( \frac{4}{4 + {\varphi_n^c}n^{-2/3}} \right)^{\lfloor n^{2/3} \rfloor} \\
	&\leq C' \exp(-c' \varphi_n^c).  
\end{align*}
Combining the large deviation bound above with \eqref{eq:ljann} yields
$$ \lambda_j = 4 + O \left( \frac{\varphi_n^c}{n^{2/3}} \right) $$
uniformly for all $n - \varphi_n^c \leq j \leq n$ with probability $1 - \exp(-\Omega(\varphi_n^c))$.  Therefore, by the triangle inequality and \eqref{eq:4gnk}
$$ \sup_{n - \varphi_n^c \leq j \leq n} |\lambda_j - \gamma_j| = O \left( \frac{\varphi_n^c}{n^{2/3}} \right) $$
with probability $1 - \exp(-\Omega(\varphi_n^c))$.  
\end{enumerate}
Since the cases above cover all $1 \leq j \leq n$, the proof of Theorem \ref{thm:sc:rigidity} is complete.  
\end{proof}

\appendix

\section{Proof of Theorem \ref{thm:wigner:rigidity}} \label{sec:wigner:rigidity:proof}

This section is devoted to the proof of Theorem \ref{thm:wigner:rigidity}.  We will need the following version of \cite[Theorem 3.6]{LY}.  

\begin{theorem}[\cite{LY}] \label{thm:LY:rigidity}
Let $M_n = \frac{1}{\sqrt{n}} W_n$ be a real symmetric Wigner matrix where $W_n = (w_{ij})_{1 \leq i,j \leq n}$.  Suppose there exists constants $C_1,c_1>0$ and $0 < \eps < 1/2$ such that
$$ \sup_{1 \leq i < j \leq n} \E[w_{ij}^4] \leq C_1 \quad\text{and}\quad \sup_{1 \leq i < j \leq n} \Prob( |w_{ij}| > n^{1/2 - \eps} ) \leq e^{-n^{c_1}}. $$
Then there exists constants $c>0$ and $n_0$ (which depend only on $C_1, \eps$, and $\sigma$ from Definition \ref{def:wigner}) such that the event
$$ \bigcup_{j=1}^n \left\{ |\lambda_j(M_n) - \eta_j| \leq (\log n)^{c \log \log n} n^{-2/3} [\min\{j,n-j+1\}]^{-1/3} \right\} $$
holds with overwhelming probability for any $n > n_0$.  
\end{theorem}

\begin{proof}[Proof of Theorem \ref{thm:wigner:rigidity}]
Set $\eps_n := n^{1/2 - \eps}$; we remind the reader that $0 < \eps < 1/2$ and hence $\eps_n \rightarrow \infty$ as $n \rightarrow \infty$.  We begin with a truncation.  Let
$$ \hat{w}_{ij} := w_{ij} \indicator{|w_{ij}| \leq \eps_n} \quad \text{for}\quad 1 \leq i \leq j \leq n. $$
We define the values $ \mu_{ij} := \E \hat{w}_{ij}$ and $ \tau^2_{ij} := \E[w_{ij}^2] - \E[ \hat{w}_{ij}^2 ]$ for $1 \leq i \leq j \leq n$.  Then by \eqref{eq:4mom}, we have
\begin{align}
	\sup_{1 \leq i < j \leq n} |\mu_{ij}| \leq \frac{C_1}{\eps_n^3}, \quad \sup_{1 \leq i \leq n} |\mu_{ii}| \leq \frac{\sigma^2}{\eps_n}  \label{eq:supbd1}\\
	\sup_{1 \leq i < j \leq n} \tau_{ij}^2 \leq \frac{C_1}{\eps_n^2}, \quad \sup_{1 \leq i \leq n} \tau_{ii}^2 \leq \sigma^2.  \label{eq:supbd2}
\end{align}

For $1 \leq i \leq j \leq n$ define the random variable $\tilde{w}_{ij}$ as a mixture of
\begin{itemize}
\item $\hat{w}_{ij}$ with probability $1 - \frac{|\mu_{ij}|}{\eps_n} - \frac{\tau_{ij}^2}{\eps_n^2}$ and
\item $z_{ij}$ with probability $\frac{|\mu_{ij}|}{\eps_n} + \frac{\tau_{ij}^2}{\eps_n^2}$,
\end{itemize}
where $z_{ij}, 1 \leq i \leq j \leq n$ are independent Bernoulli random variables independent of $W_n$.  Set $\tilde{w}_{ji} = \tilde{w}_{ij}$ for $1 \leq i < j \leq n$.  Let $\tilde{W}_n = (\tilde{w}_{ij})_{1 \leq i,j \leq n}$ and $\tilde{M}_n = \frac{1}{\sqrt{n}} \tilde{W}_n$.  

We now show that there exists Bernoulli random variables $z_{ij}$ such that $\tilde{M}_n$ is a real symmetric Wigner matrix that satisfies  
\begin{equation} \label{eq:supcheck1}
	\sup_{1 \leq i \leq j \leq n} |\tilde{w}_{ij}| \leq n^{1/2 - \eps/2} \quad \text{almost surely} 
\end{equation}
and 
\begin{equation} \label{eq:supcheck2}
	\sup_{1 \leq i < j \leq n} \E[ \tilde{w}_{ij}^4] \leq 513 C_1.
\end{equation}

In particular, we will construct $z_{ij}$ to be a Bernoulli random variable, symmetric about its mean, such that its mean and second moment satisfy 
\begin{align}
	0 &= \mu_{ij} \left( 1 - \frac{|\mu_{ij}|}{\eps_n} - \frac{\tau_{ij}^2}{\eps_n^2} \right) + \E[z_{ij}] \left(  \frac{|\mu_{ij}|}{\eps_n} + \frac{\tau_{ij}^2}{\eps_n^2} \right) \label{eq:zij1}	\\
	\E[w_{ij}^2] &= (\E[w_{ij}^2] - \tau_{ij}^2) \left( 1 - \frac{|\mu_{ij}|}{\eps_n} - \frac{\tau_{ij}^2}{\eps_n^2} \right) + \E[z_{ij}^2] \left(  \frac{|\mu_{ij}|}{\eps_n} + \frac{\tau_{ij}^2} {\eps_n^2} \right) \label{eq:zij2}
\end{align}
for $1 \leq i \leq j \leq n$.  We first note that, by definition of $\tilde{w}_{ij}$, we only need to consider the case when $\frac{|\mu_{ij}|}{\eps_n} + \frac{\tau_{ij}^2}{\eps_n^2} > 0$.  Suppose $a_{ij},b_{ij}$ are real numbers that satisfy
\begin{align*}
	0 &= \mu_{ij} \left( 1 - \frac{|\mu_{ij}|}{\eps_n} - \frac{\tau_{ij}^2}{\eps_n^2} \right) + a_{ij} \left(  \frac{|\mu_{ij}|}{\eps_n} + \frac{\tau_{ij}^2}{\eps_n^2} \right)	\\
	\E[w_{ij}^2] &= (\E[w_{ij}^2] - \tau_{ij}^2) \left( 1 - \frac{|\mu_{ij}|}{\eps_n} - \frac{\tau_{ij}^2}{\eps_n^2} \right) + b_{ij}^2 \left(  \frac{|\mu_{ij}|}{\eps_n} + \frac{\tau_{ij}^2}{\eps_n^2} \right).
\end{align*}
From the first equation, we obtain
\begin{equation} \label{eq:aij}
	|a_{ij}| \left(  \frac{|\mu_{ij}|}{\eps_n} + \frac{\tau_{ij}^2}{\eps_n^2} \right) \leq |\mu_{ij}|. 
\end{equation}
From the second equation, we have
\begin{equation} \label{eq:bij}
	b_{ij}^2 \left(  \frac{|\mu_{ij}|}{\eps_n} + \frac{\tau_{ij}^2}{\eps_n^2} \right) = (\E[w_{ij}^2] - \tau_{ij}^2) \left(  \frac{|\mu_{ij}|}{\eps_n} + \frac{\tau_{ij}^2}{\eps_n^2} \right) + \tau_{ij}^2 \geq \tau_{ij}^2. 
\end{equation}
We now note that $\tau_{ij}^2 \geq \eps_n |\mu_{ij}|$ by definition of $\hat{w}_{ij}$ and hence
$$ \tau_{ij}^4 + \tau_{ij}^2|\mu_{ij}| \eps_n - |\mu_{ij}|^2 \eps_n^2 \geq 0. $$
It then follows that
\begin{equation} \label{eq:taumu}
	\tau_{ij}^2 \geq \frac{|\mu_{ij}|^2}{\left(  \frac{|\mu_{ij}|}{\eps_n} + \frac{\tau_{ij}^2}{\eps_n^2} \right)}. 
\end{equation}
Combining \eqref{eq:aij}, \eqref{eq:bij}, and \eqref{eq:taumu}, we obtain
$$ b_{ij}^2 \left(  \frac{|\mu_{ij}|}{\eps_n} + \frac{\tau_{ij}^2}{\eps_n^2} \right) \geq \tau_{ij}^2 \geq \frac{\mu_{ij}^2}{\left(  \frac{|\mu_{ij}|}{\eps_n} + \frac{\tau_{ij}^2}{\eps_n^2} \right)} \geq a_{ij}^2 \left(  \frac{|\mu_{ij}|}{\eps_n} + \frac{\tau_{ij}^2}{\eps_n^2} \right) $$
and hence $b_{ij}^2 \geq a_{ij}^2$.  We can now define 
$$ z_{ij} := \left\{
     \begin{array}{lr}
       a_{ij} + \sqrt{b_{ij}^2 - a_{ij}^2} & \text{ with probability } 1/2 \\
       a_{ij} - \sqrt{b_{ij}^2 - a_{ij}^2} & \text{ with probability }1/2
     \end{array}
   \right. .$$
It is straightforward to verify that $z_{ij}$ has mean $a_{ij}$ and second moment $b_{ij}^2$.  

By construction $\tilde{M}_n$ is a real symmetric Wigner matrix.  We now verify \eqref{eq:supcheck1} and \eqref{eq:supcheck2}.  By solving equations \eqref{eq:zij1} and \eqref{eq:zij2} for $\E[z_{ij}]$ and $\E[z_{ij}^2]$ and applying the bounds \eqref{eq:supbd1} and \eqref{eq:supbd2}, it follows that $|z_{ij}| \leq 4 \eps_n$.  Thus we conclude that \eqref{eq:supcheck1} holds for $n$ sufficiently large.  

We also have for $1 \leq i < j \leq n$
\begin{align*}
	\E[\tilde{w}_{ij}^4] &= \E[\hat{w}_{ij}^4] \left( 1 - \frac{|\mu_{ij}|}{\eps_n} - \frac{\tau_{ij}^2}{\eps_n^2} \right) + \E[z_{ij}^4] \left(  \frac{|\mu_{ij}|}{\eps_n} + \frac{\tau_{ij}^2}{\eps_n^2} \right) \\
	& \leq C_1 + (4 \eps_n)^4 2 \frac{C_1}{\eps_n^4} \\
	& \leq 513 C_1
\end{align*}
by \eqref{eq:supbd1} and \eqref{eq:supbd2}.  This verifies \eqref{eq:supcheck2} and hence $\tilde{M}_n$ satisfies the conditions of Theorem \ref{thm:LY:rigidity}.  

By Theorem \ref{thm:LY:rigidity}, there exists a constant $c>0$ such that the event
$$ \Omega_n := \bigcup_{j=1}^n \left\{ |\lambda_j(\tilde{M}_n) - \eta_j| \leq (\log n)^{c \log \log n} n^{-2/3} [\min\{j,n-j+1\}]^{-1/3} \right\}  $$
holds with overwhelming probability.  Thus we obtain
\begin{align*}
	\Prob &\left(  \exists j : |\lambda_j(M_n) - \eta_j| \geq (\log n)^{c \log \log n} n^{-2/3} [\min\{j,n-j+1\}]^{-1/3} \right) \\
	& \qquad \qquad \leq \Prob(\Omega_n^C) + \Prob(M_n \neq \tilde{M}_n).  
\end{align*}

The proof of Theorem \ref{thm:wigner:rigidity} is now complete by noting that
\begin{align*}
	\Prob( M_n \neq \tilde{M}_n) &\leq \sum_{i,j=1}^n \Prob(|w_{ij}| > \eps_n) + \sum_{i,j=1}^n \left(  \frac{|\mu_{ij}|}{\eps_n} + 
\frac{\tau_{ij}^2}{\eps_n^2} \right) \\
	& \leq \frac{2}{\eps_n^4} \sum_{1 \leq i < j \leq n} \E[ w_{ij}^4 \indicator{|w_{ij}| > \eps_n}] + \frac{2}{\eps_n^2} 
\sum_{i =1}^n \E[ w_{ii}^2 \indicator{|w_{ii}| > \eps_n}].
\end{align*}
\end{proof}

\section{Proof of Lemma \ref{lemma:depclt}} \label{sec:depclt}

In order to prove Lemma \ref{lemma:depclt}, we use the central limit theorem for martingale difference sequences.  

\begin{theorem}[Theorem 35.12 of \cite{B}] \label{thm:clt}
For each $N$, suppose $Z_{N1}, Z_{N2}, \ldots, Z_{Nr_N}$ is a real martingale difference sequence with respect to the increasing 
$\sigma$-field $\{\mathcal{F}_{N,j}\}$ having second moments.  If as $N \rightarrow \infty$, 
\begin{equation} \label{clt_variance}
        \sum_{j=1}^{r_N} \E(Z^2_{Nj} \mid \mathcal{F}_{N,j-1}) \overset{P}{\longrightarrow} v^2
\end{equation}
where $v^2$ is a positive constant, and for each $\epsilon > 0$,
\begin{equation} \label{clt_lf}
        \sum_{j=1}^{r_N} \E(Z^2_{Nj} \indicator{|Z_{Nj}| \geq \epsilon}) \rightarrow 0
\end{equation}
then
\begin{equation*}
        \sum_{j=1}^{r_N} Z_{Nj} \overset{\mathcal{L}}{\longrightarrow} \mathcal{N}(0, v^2).
\end{equation*}
\end{theorem}

We will also need a number of computations, which we collect in the following lemma.

\begin{lemma}[Computations] \label{lemma:comp}
For each $n \geq 1$, let $(\xi^{(n)}_1, \ldots, \xi^{(n)}_k)$ be a discrete random sample on $[n]:=\{1,2,\ldots, n\}$, where $k = k(n)$ is a positive integer sequence.  Define $\zeta_i^{(n)} := \frac{\xi_i^{(n)}}{n}$ for each $n \geq 1$ and $i=1,2,\ldots,k$.  Let $g: [0,1] \rightarrow \R$ be a bounded function.  Then there exits a constant $C>0$ (depending only on the function $g$) such that
\begin{equation} \label{eq:comp2}
	\left| \E \left[ \E_{j-1} g(\zeta_{j}^{(n)}) \right]^2 - \left[ \E g(\zeta_{1}^{(n)}) \right]^2 \right| \leq \frac{C}{n-j+1}
\end{equation}
and
\begin{equation} \label{eq:comp4}
	\left| \E \left[ \E_{j-1} g(\zeta_{j}^{(n)}) \right]^4 - \left[ \E g(\zeta_{1}^{(n)}) \right]^4 \right| \leq \frac{C}{n-j+1}.
\end{equation}
\end{lemma}
\begin{proof}
We write
\begin{align*}
	\E \left[ \E_{j-1} g(\zeta_{j}^{(n)}) \right]^2 &= \E \sum_{s,t \notin \{ \xi_1^{(n)}, \ldots, \xi_{j-1}^{(n)} \}} \frac{g(s/n) g(t/n)}{(n-j+1)^2} \\
	&= \sum_{S \subset [n]; |S| = j-1} \frac{1}{n(n-1) \cdots (n-j+2)} \sum_{s,t \notin S} \frac{g(s/n) g(t/n)}{(n-j+1)^2} \\
	&= \sum_{s,t=1}^n \frac{g(s/n) g(t/n)}{(n-j+1)^2}  \sum_{|S|=j-1; s,t \notin S} \frac{1}{n(n-1) \cdots (n-j+2)} \\
	&= \sum_{s=1}^n \frac{g^2(s/n)}{(n-j+1)^2} \sum_{|S|=j-1; s \notin S} \frac{1}{n(n-1) \cdots (n-j+2)}  \\
		&\qquad + \sum_{s \neq t} \frac{g(s/n) g(t/n)}{(n-j+1)^2} \sum_{|S|=j=1;s,t \notin S} \frac{1}{n(n-1) \cdots (n-j+2)} \\
	&= \frac{1}{n-j+1} \E g^2(\zeta_1^{(n)}) + \frac{n-j+2}{n-j+1} \E \left[ g(\zeta_1^{(n)}) g(\zeta_2^{(n)}) \right] \\
	&= \E \left[ g(\zeta_1^{(n)}) g(\zeta_2^{(n)}) \right] + O_g \left( \frac{1}{n-j+1}\right),
\end{align*}
where the set $S$ in the sums above is an ordered set.  We now note that 
\begin{align*}
	\E \left[ g(\zeta_1^{(n)}) g(\zeta_2^{(n)}) \right] &= \sum_{s \neq t} \frac{g(s/n) g(t/n)}{n(n-1)} \\
	&= \sum_{s \neq t} \frac{g(s/n) g(t/n)}{n^2} + O_g \left( \frac{1}{n} \right) \\
	&= \sum_{s,t=1}^n \frac{g(s/n) g(t/n)}{n^2} + O_g \left( \frac{1}{n} \right).
\end{align*}
Combing the estimates above yields \eqref{eq:comp2}.  For \eqref{eq:comp4}, we write
\begin{align*}
	\E \left[ \E_{j-1} g(\zeta_j^{(n)}) \right]^4 &= \E \sum_{t_1, \ldots, t_4 \notin \{ \xi_1^{(n)}, \ldots, \xi_{j-1}^{(n)} \}} \frac{ g(t_1/n) \cdots g(t_4/n)}{(n-j+1)^4} \\
	&= \sum_{S \subset [n]; |S|=j-1} \frac{1}{n(n-1)\cdots(n-j+2)} \sum_{t_1, \ldots, t_4 \notin S} \frac{g(t_1/n) \cdots g(t_4/n)}{(n-j+1)^4} \\
	&= \sum_{t_1,\ldots,t_4=1}^n \frac{g(t_1/n) \cdots g(t_4/n)}{(n-j+1)^4} \sum_{|S|=j-1; t_1, \ldots, t_4 \notin S} \frac{1}{n(n-1)\cdots (n-j+2)},
\end{align*}
where the set $S$ in the sums above is an ordered set.  We now consider several cases where $t_1, \ldots, t_4$ are not distinct.  
\begin{enumerate}
\item When the first sum is over $t_1 = t_2$ and $t_1, t_3, t_4$ are distinct, we obtain
\begin{align*}
	\sum_{t_1, t_3, t_4} & \frac{g(t_1/n) \cdots g(t_4/n)}{(n-j+1)^4} \sum_{|S|=j-1; t_1,t_3,t_4 \notin S} \frac{1}{n(n-1)\cdots (n-j+2)} \\
	&= \sum_{t_1,t_3,t_4} \frac{ g(t_1/n)^2 g(t_3/n) g(t_4/n)}{(n-j+1)^4} \frac{(n-3) \cdots (n-j-1)}{n(n-1) \cdots (n-j+2)} \\
	&= \E\left[ g^2(\zeta_1^{(n)}) g(\zeta_2^{(n)}) g(\zeta_3^{(n)}) \right] \frac{(n-j)(n-j-1)}{(n-j+1)^3} \\
	&= O_g \left( \frac{1}{n-j+1} \right).
\end{align*} 
\item When the sum is over $t_1=t_2=t_3 \neq t_4$, we have
\begin{align*}
	\sum_{t_1 \neq  t_4} & \frac{g(t_1/n)^3 g(t_4/n)}{(n-j+1)^4} \sum_{|S|=j-1; t_1, t_4 \notin S} \frac{1}{n(n-1)\cdots (n-j+2)} \\
	&= \E\left[ g(\zeta_1^{(n)})^3 g(\zeta_2^{(n)}) \right] \frac{(n-j+1)(n-j)}{(n-j+1)^4} \\
	&= O_g \left( \frac{1}{(n-j+1)^2} \right).
\end{align*}
\item When the sum is over $t_1=t_2=t_3=t_4$, we obtain
\begin{align*}
	\sum_{t_1=1}^n &\frac{g(t_1/n)^4}{(n-j+1)^4} \sum_{|S| = j-1; t_1 \notin S} \frac{1}{n(n-1)\cdots (n-j+2)} \\
	&= \E \left[ g(\zeta_1^{(n)}) ^4 \right] \frac{1}{(n-j+1)^3} \\
	&= O_g \left( \frac{1}{(n-j+1)^3} \right).
\end{align*}
\end{enumerate}
Combining the above bounds yields
\begin{align*}
	\E &\left[ \E_{j-1} g(\zeta_j^{(n)}) \right]^4 \\
	&= \sum_{t_1,\ldots,t_4 \text{ distinct}} \frac{g(t_1/n) \cdots g(t_4/n)}{(n-j+1)^4} \sum_{|S|=j-1; t_1, \ldots, t_4 \notin S} \frac{1}{n(n-1)\cdots (n-j+2)}  \\
	&\qquad + O_g \left( \frac{1}{n-j+1} \right). 
\end{align*}
When $t_1, \ldots, t_4$ are distinct, we can compute the inside sum and obtain
\begin{align*}
	\E \left[ \E_{j-1} g(\zeta_j^{(n)}) \right]^4 &= \E \left[ g(\zeta_1^{(n)}) \cdots g(\zeta_4^{(n)}) \right] \frac{(n-j)(n-j-1)(n-j-2)}{(n-j+1)^3} \\
	&\qquad+ O_g \left( \frac{1}{n-j+1} \right) \\
	&= \E \left[ g(\zeta_1^{(n)}) \cdots g(\zeta_4^{(n)}) \right] + O_g \left( \frac{1}{n-j+1} \right).
\end{align*}
Lastly, we note that
\begin{align*}
	\E \left[ g(\zeta_1^{(n)}) \cdots g(\zeta_4^{(n)}) \right] &= \sum_{t_1, \ldots, t_4 \text{ distinct}} \frac{g(t_1/n) \cdots g(t_4/n)}{n (n-1) (n-2)(n-3)} \\
	&= \sum_{t_1, \ldots, t_4 \text{ distinct}} \frac{g(t_1/n) \cdots g(t_4/n)}{n^4} + O_g \left( \frac{1}{n} \right) \\
	&= \sum_{t_1, \ldots, t_4=1}^n \frac{g(t_1/n) \cdots g(t_4/n)}{n^4} + O_g \left( \frac{1}{n} \right) \\
	&= \left[ \E g(\zeta_1^{(n)}) \right]^4 + O_g \left( \frac{1}{n} \right),
\end{align*}
and the proof of Lemma \ref{lemma:comp} is complete.  
\end{proof}

\begin{proof}[Proof of Lemma \ref{lemma:depclt}]
We will use Theorem \ref{thm:clt} to prove Lemma \ref{lemma:depclt}.  We write 
$$ \alpha_{n,k} \sum_{i=1}^k [f({\zeta^{(n)}_i}) - \E f({\zeta^{(n)}_i})] = \sum_{j=1}^k Z_{n,j} $$
where
$$ Z_{n,j} := \alpha_{n,k} \sum_{i=1}^k \left[ \E_j f({\zeta^{(n)}_i}) - \E_{j-1} f({\zeta^{(n)}_i}) \right], $$
$\E_j$ denotes expectation with respect to the $\sigma$-algebra $\mathcal{F}_{n,j}$, and $\mathcal{F}_{n,j} = \sigma(\xi^{(n)}_1, \ldots, \xi^{(n)}_j)$.  

By considering the cases when $i < j$, $i=j$, and $i > j$, we have that
$$ Z_{n,j} = \alpha_{n,k} \left[ f(\zeta^{(n)}_j) - \E_{j-1} f(\zeta^{(n)}_j) + (k-j) \left( \E_{j} f(\zeta^{(n)}_{j+1}) - \E_{j-1} f(\zeta^{(n)}_j) \right) \right].  $$
We now compute
\begin{align*}
	\E_j f(\zeta^{(n)}_{j+1}) &= \sum_{t \notin \{ \xi^{(n)}_1, \ldots, \xi^{(n)}_j \}} f(t/n) \frac{1}{n-j} \\
	&= \left( 1 + \frac{1}{n-j}\right) \E_{j-1} f(\zeta^{(n)}_j) - \frac{1}{n-j} f(\zeta^{(n)}_j).
\end{align*}
Thus,
$$ Z_{n,j} = \alpha_{n,k} \frac{n-k}{n-j} \left[ f(\zeta^{(n)}_j) - \E_{j-1} f(\zeta^{(n)}_j) \right].  $$

Since $f$ is bounded and $\alpha_{n,k} = o(1)$ it follows that $Z_{n,j}= o(1)$ uniformly for $j=1,2,\ldots, k$.  So the events $\{|Z_{n,j}| > \epsilon \}$ are empty for $n$ sufficiently large.  Thus \eqref{clt_lf} holds.  

We now verify \eqref{clt_variance} and compute the limiting variance.  We note that 
\begin{equation} \label{eq:ejz2}
	\sum_{j=1}^k \E_{j-1} [Z_{n,j}^2] = \alpha_{n,k}^2 (n-k)^2 \sum_{j=1}^k \frac{1}{(n-j)^2} \left[ E_{j-1} f^2(\zeta^{(n)}_j) - \left( E_{j-1} f(\zeta^{(n)}_j) \right)^2 \right].
\end{equation}
We will show that
\begin{equation} \label{eq:comp:show1}
	\alpha_{n,k}^2 (n-k)^2 \E \left| \sum_{j=1}^k \frac{1}{(n-j)^2} \left[ \E_{j-1} f^2(\zeta_{j}^{(n)}) - \E f^2(\zeta_1^{(n)}) \right] \right| \longrightarrow 0
\end{equation}
and
\begin{equation} \label{eq:comp:show2}
	\alpha_{n,k}^2 (n-k)^2 \E \left| \sum_{j=1}^k \frac{1}{(n-j)^2} \left[ \left( \E_{j-1} f(\zeta_j^{(n)}) \right)^2  - \left( \E f(\zeta_1^{(n)}) \right)^2 \right] \right| \longrightarrow 0
\end{equation}
as $n \rightarrow \infty$.  

For \eqref{eq:comp:show1}, it suffices to prove that
$$ \alpha_{n,k}^2 (n-k)^2 \sum_{j=1}^k \frac{1}{(n-j)^2} \sqrt{ \E \left| \E_{j-1} f^2(\zeta_{j}^{(n)}) - \E f^2(\zeta_1^{(n)}) \right|^2 } \longrightarrow 0. $$
By Lemma \ref{lemma:comp}, we have that
\begin{align*}
	 \alpha_{n,k}^2 & (n-k)^2 \sum_{j=1}^k \frac{1}{(n-j)^2}  \sqrt{\E \left| \E_{j-1} f^2(\zeta_{j}^{(n)}) - \E f^2(\zeta_1^{(n)}) \right|^2} \\
	 &=  \alpha_{n,k}^2 (n-k)^2 \sum_{j=1}^k \frac{1}{(n-j)^2} \sqrt{ \E \left[ \E_{j-1} f^2(\zeta_{j}^{(n)}) \right]^2 - \left[ \E f^2(\zeta_1^{(n)}) \right]^2 } \\
	&\leq  \alpha_{n,k}^2 (n-k)^2 \sum_{j=1}^k \frac{\sqrt{C}}{(n-j)^{2.5}}  \\
	& \leq \frac{ \sqrt{C}}{\sqrt{n-k}} \alpha_{n,k}^2 (n-k)^2 \sum_{j=1}^k \frac{1}{(n-j)^{2}}  \\
	&\leq \frac{ \sqrt{C}}{\sqrt{n-k}} \alpha_{n,k}^2 (n-k)^2 \frac{k}{(n-1)(n-k)} \longrightarrow 0.
\end{align*}
Here the last inequality comes from a comparison argument between $\sum_{j=1}^k \frac{1}{(n-j)^2}$ and an appropriate integral.  This verifies \eqref{eq:comp:show1}.  The proof of \eqref{eq:comp:show2} is similar and uses \eqref{eq:comp4}.  

Using \eqref{eq:comp:show1} and \eqref{eq:comp:show2}, we have that
$$ \sum_{j=1}^k \E_{j-1} [Z_{n,j}^2] - \var[ f(\zeta_1^{(n)})] \alpha_{n,k}^2 (n-k)^2 \sum_{j=1}^k \frac{1}{(n-j)^2} \longrightarrow 0 $$
in probability as $n \rightarrow \infty$.  However, a comparison argument verifies that
$$ \lim_{n \rightarrow \infty} \alpha_{n,k}^2 (n-k)^2 \sum_{j=1}^k \frac{1}{(n-j)^2} = 1 $$
and the proof of Lemma \ref{lemma:depclt} is complete.  
\end{proof}

\end{document}